\documentclass[a4paper]{amsart}
\usepackage{fullpage}
\usepackage[english]{babel}
\usepackage[cp1250]{inputenc}
\usepackage[T1]{fontenc}
\usepackage{amsmath,amsfonts,amssymb,amsthm}
\usepackage{graphicx}
\usepackage{color}
\usepackage{mathtools}

\newtheorem{theorem}{Theorem}[section]
\newtheorem{lemma}[theorem]{Lemma}
\newtheorem*{lemma*}{Lemma}
\newtheorem{proposition}[theorem]{Proposition}

\newtheorem{step}{Step}

\theoremstyle{definition}
\newtheorem*{definition}{Definition}
\newtheorem*{convention}{Convention}

\theoremstyle{remark}
\newtheorem*{remark}{Remark}

\begin{document}

\title{On the definition of homological critical value}
\author{Dejan Govc}
\address{Artificial Intelligence Laboratory, Jo\v{z}ef Stefan Institute\\
Jamova 39, 1000 Ljubljana, Slovenia}
\email{dejan.govc@gmail.com}
\date{\today}

\begin{abstract}
We point out that there is a problem with the definition of homological critical value (as defined in the widely cited paper \cite{stability} by Cohen-Steiner, Edelsbrunner and Harer). Under that definition, the critical value lemma of \cite{stability} in fact fails. We provide several counterexamples and a definition (due to Bubenik and Scott \cite{categorification}) we feel should be preferred and under which the critical value lemma does indeed hold. One of the counterexamples we have found is a height function on a compact smooth manifold. In the end we prove that, despite all this, a modified version of the critical value lemma remains valid under the original definition.
\end{abstract}

\maketitle

\section{Introduction}

As already noticed by Bubenik and Scott in \cite{categorification}, the definition of homological critical value given by Cohen-Steiner, Edelsbrunner and Harer in \cite{stability} is problematic. To provide some justification for their own version of the definition, Bubenik and Scott \cite[footnote of Definition 4.3]{categorification} present a counterexample to the critical value lemma of \cite{stability} under the original definition of \cite{stability}. Their counterexample, however, is a discontinuous function. The aim of this paper is to show that the definition is problematic also in the continuous (or even smooth) setting.

We first show that under the definition of homological critical value in \cite{stability}, the critical value lemma of \cite{stability} in fact fails: in the second section, we present some simple counterexamples with increasingly nice properties. Next, we show that there is a compact smooth manifold for which the critical value lemma fails. This requires some proving, thus we have devoted the whole third section to this example. In the fourth section we present a version of the critical value lemma that holds even under the problematic definition.

\subsection{Definitions and conventions}

We need to carefully distinguish between two definitions of homological critical value. This requires us to establish some new terminology. We first describe the currently accepted terminology, and then suggest the changes in terminology we believe are necessary.

In \cite{stability}, the authors define homological critical value as follows:

\begin{definition}
Let $X$ be a topological space and $f$ a real function on $X$. A {\em homological critical value} of $f$ is a real number $a$ for which there exists an integer $k$ such that for all sufficiently small $\epsilon>0$ the map $H_k(f^{-1}(-\infty,a-\epsilon])\to H_k(f^{-1}(-\infty,a+\epsilon])$ induced by inclusion is not an isomorphism.
\end{definition}

Here $H_k$ denotes the $k$-th singular homology (possibly with coefficients in a field, as is usually assumed in the context of persistence). The main purpose of this definition is to establish the critical value lemma, which says the following:

\begin{lemma*}
Suppose the function $f:X\to\mathbb R$ has no homological critical values on the closed interval $[x,y]$. Then the inclusion $f^{-1}(-\infty,x]\hookrightarrow f^{-1}(-\infty,y]$ induces isomorphisms on all homology groups.
\end{lemma*}

As mentioned, under the above definition, the critical value lemma is in fact false. This problem is solved easily by replacing the above definition with the following one, suggested by Bubenik and Scott \cite[footnote of Definition 4.3]{categorification}, which we state in a slightly less general form as follows:

\begin{definition}
Let $X$ be a topological space and $f$ a real function on $X$. A real number $a$ is a {\em homological regular value} of the function $f$ if there exists an $\epsilon>0$ such that for each pair of real numbers $x<y$ on the interval $(a-\epsilon,a+\epsilon)$, the inclusion $f^{-1}(-\infty,x]\hookrightarrow f^{-1}(-\infty,y]$ induces isomorphisms on all homological groups. A real number $a$ that is not a homological regular value of $f$ is called a {\em homological critical value} of $f$.
\end{definition}

For this definition, the critical value lemma does indeed hold, as Bubenik and Scott show (in a more general form) in \cite[Lemma 4.4]{categorification}. This leaves us with two non-equivalent definitions of homological critical value. To distinguish between the two, we adopt the following

\begin{convention} Homological critical values from the first (problematic) definition \cite{stability} will be called {\em symmetric homological critical values} of $f$. Homological critical values from the second definition will simply be called {\em homological critical values}, i.e. we accept the second definition as the preferred one. A real number that is not a symmetric homological critical value of $f$ will be called a {\em symmetric homological regular value}.
\end{convention}

\section{Some simple counterexamples}

We will now show that under the symmetric definition of homological critical value, the critical value lemma fails.

\subsection{A continuous counterexample}

The following counterexample is inspired by the discontinuous counterexample given by Bubenik and Scott in \cite{categorification}. Define a topological space $X\subseteq\mathbb R^2$ (equipped with the subspace topology) as follows:
\begin{equation*}
X = \{0\}\times[-1,1]\cup(0,1)\times(0,1]\cup\{1\}\times[0,1]
\end{equation*}

(Throughout this paper, we adopt the convention that $\times$ binds more strongly than $\cup$.) Define a function $f:X\to\mathbb R$ on this space by the formula $f(x,y)=y$, i.e. the height function. The various sublevel sets $f^{-1}(-\infty,a]$ look as follows. (Note that $X$ looks exactly like the sublevel set pictured for $a>0$, but with a square instead of rectangle.)

\hbox{}

\begin{center}
\includegraphics{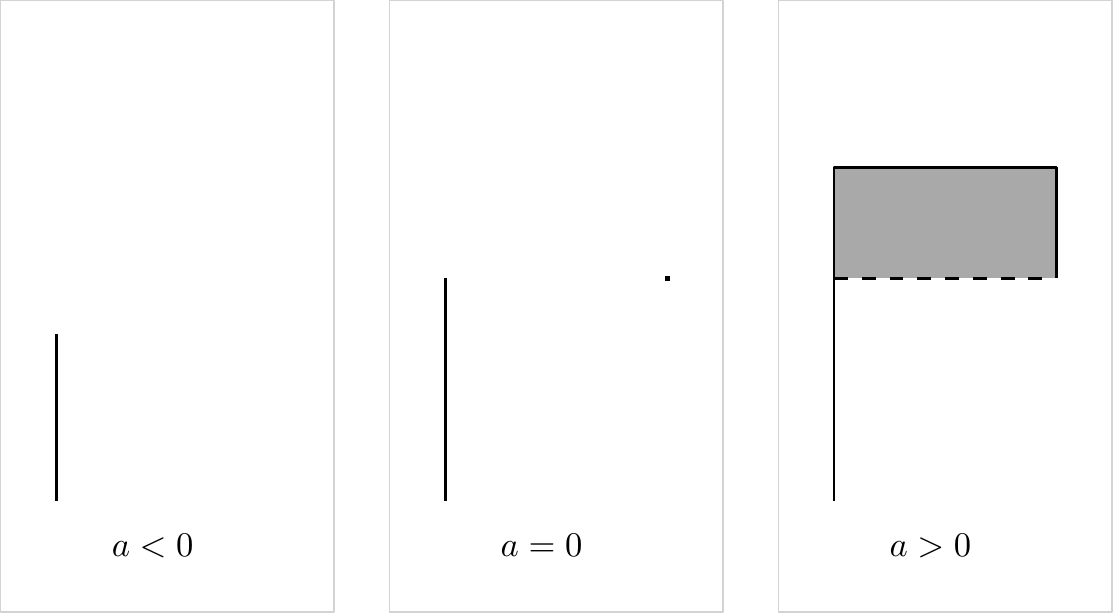}
\end{center}

We claim that $f:X\to\mathbb R$ has no symmetric homological critical values in the interval $(-1,1)$. To see this, first notice that $f^{-1}(-\infty,a]$ is contractible for all $a\in[-1,1]\setminus\{0\}$. This implies that every $a\in(-1,0)\cup(0,1)$ has a neighborhood $(a-\delta,a+\delta)$ such that for $\epsilon\in(0,\delta)$ the induced homomorphisms on the $k$-th homology groups are trivial for all $k\geq 1$. Since the corresponding homology groups are also trivial, they are isomorphisms. Furthermore, the induced homomorphisms on zeroth homology groups are isomorphisms as well, because the underlying spaces are path-connected.

For $a=0$ the situation is similar: taking any $\epsilon\in(0,1)$, the spaces $f^{-1}(-\infty,-\epsilon]$ and $f^{-1}(-\infty,\epsilon]$ are contractible and thus the inclusion $f^{-1}(-\infty,-\epsilon]\hookrightarrow f^{-1}(-\infty,\epsilon]$ induces isomorphisms on all homology groups. 

This shows that there are indeed no symmetric homological critical values in $(-1,1)$. Therefore there are no symmetric homological critical values in $[-\frac12,0]$ either. Despite this, $f^{-1}(-\infty,-\frac12]$ has one path component, while $f^{-1}(-\infty,0]$ has two, implying that their zeroth homology groups cannot be isomorphic. This contradicts the critical value lemma as presented in \cite{stability}.

Note that we can make the critical value lemma fail on higher homology groups just as easily. For example, it fails (by an analogous argument) on $H_1$ if we define $X\subseteq\mathbb R^3$ (using, of course, the natural identification of $\mathbb R^3$ with $\mathbb R^2\times\mathbb R$) by
\begin{equation*}
X = \{(-1,0)\}\times[-1,1]\cup S^1\times[0,1]\cup D^2\times(0,1]
\end{equation*}
where $S^1,D^2\subseteq\mathbb{R}^2$ denote the unit circle and the closed unit disk, respectively. The function $f:X\to\mathbb R$ for which the critical value lemma fails is again a height function, this time defined by $f(x,y,z)=z$.

\subsection{A compact counterexample}

A natural question that arises after seeing the counterexample just presented is: can we make the underlying topological space compact? The answer is: indeed, we can. To see this, let
\begin{equation*}
K=\left\{\left(x,\sin\frac{\pi}x\right)\big|\;x\in(0,1]\right\}\cup\{0\}\times[-1,1]\subseteq\mathbb R^2
\end{equation*}
be the usual topologist's sine curve. Now, for $r\in[0,1]$, define
\begin{equation*}
K_r:=K\cup [0,r]\times[-1,1]
\end{equation*}
and let 
\begin{equation*}
X=\{(1,0)\}\times[-1,0]\cup\bigcup_{r\in[0,1]} K_r\times\{r\}\subseteq\mathbb R^3.
\end{equation*}
(Via the canonical identification $\mathbb R^2\times\mathbb R\equiv\mathbb R^3$.) As in the previous example, let $f:X\to\mathbb R$ be the height function defined by $f(x,y,z)=z$. Now, examine the sublevel set $f^{-1}(-\infty,a]$. For negative $a$, this is just the interval $\{(1,0)\}\times[-1,a]$, which is clearly contractible. For positive $a$, we get
\begin{equation*}
\{(1,0)\}\times[-1,0]\cup\bigcup_{r\in[0,a]} K_r\times\{r\}
\end{equation*}
which has an obvious deformation retraction to the ``top level'' $K_a\times\{a\}$, which is again contractible and thus, so is the sublevel set. For $a=0$, on the other hand, we get
\begin{equation*}
\{(1,0)\}\times[-1,0]\cup K\times\{0\}
\end{equation*}
which again deformation retracts to the top level, but this time the top level is a topologist's sine curve, which has two path components. As in the previous example, all this (combined with the fact that there are no symmetric homological critical values, which follows immediately from what we have just seen) means that $X$ is a counterexample to the critical value lemma. (There are again many easy ways to modify this so that the critical value lemma fails for higher homology groups: for example, rotating the topologist's sine curve around its endpoint (or end interval) and doing an analogous construction makes the critical value lemma fail for $H_1$.)

The sublevel sets of $f$ are a bit hard to draw, so instead we provide some pictures of level sets (which are precisely the horizontal slices of $X$) as seen from the top. The level sets for negative $a$ are singletons. Note that the level sets increase (in the sense of inclusion as subsets of the plane, i.e. if we forget the third coordinate) as $a$ increases. $X$ consists of these level sets stacked on top of each other. (Which gives us an easy way of visualizing it.)

\begin{center}
\includegraphics[trim= 0pt 30pt 0pt 30pt]{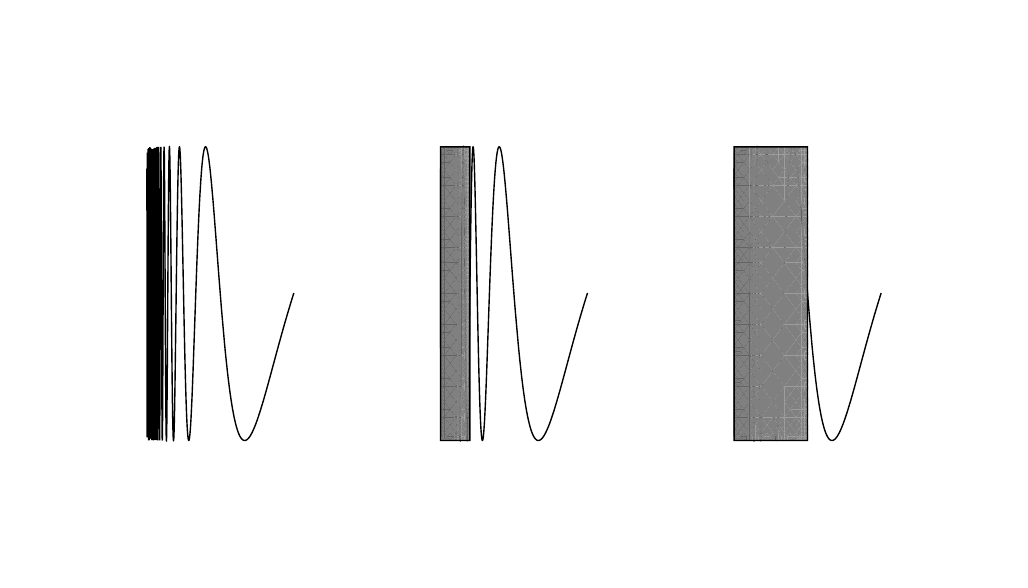}
\end{center}

Note that we could also have taken $K_r=\{(x,y)\in\mathbb R|\;d((x,y),K)\leq r\}$ and the same construction would work, basically for the same reasons. (One of the key features of such $K_r$ is that it is a contractible set for $r>0$; for the proof of this, follow the same idea as for the contractibility of sublevel sets in the next example.)

\subsection{A topological manifold counterexample}

Next, we present a topological manifold for which the critical value lemma fails under the symmetric definition. Take the topologist's sine curve $K$ from the previous example and add a square to it: $L=K\cup[-2,0]\times[-1,1]$. Now define the signed distance function (see \cite{signed}) $f:\mathbb R^2\to\mathbb R$ as follows
\begin{equation*}
f(x,y)=\begin{cases}-d((x,y),L^C);& x\in L\\
d((x,y),L);& x\notin L.
\end{cases}
\end{equation*}

\pagebreak
The set $L$ looks as follows:

\begin{center}
\includegraphics[trim= 0pt 30pt 30pt 30pt]{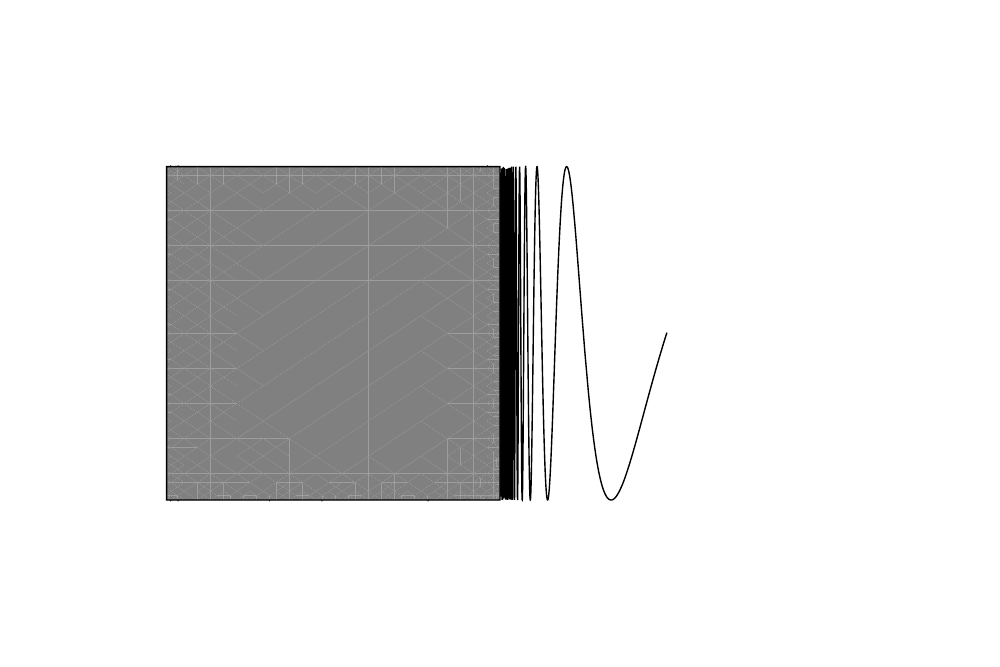}
\end{center}

This is also the sublevel set for $a=0$. The sublevel sets for $a<0$ are smaller squares inside $[-2,0]\times[-1,1]$ and the sublevel sets for $a>0$ are thickened versions of $L$.

The function $f$ is clearly continuous. Sublevel sets $f^{-1}(-\infty,a]$ are contractible for negative and positive $a$. (For the positive ones, we will prove this in a moment.) For $a=0$, the sublevel set is $L$, which has two connected components. By once again the same argument as before, this means we have found another counterexample. If we want the function to be a height function, we instead examine the graph of $f$, i.e. the set
\begin{equation*}
\Gamma_f=\{(x,y,f(x,y))|\;x,y\in\mathbb R\}\subseteq\mathbb R^3
\end{equation*}
and the function $F:\Gamma_f\to\mathbb R$ given by $F(x,y,z)=z$.

Note that $(x,y)\mapsto(x,y,f(x,y))$ is a homeomorphism from $\mathbb R^2$ to $\Gamma_f$, which shows that $\Gamma_f$ is a topological manifold. Furthermore, this homeomorphism takes the sublevel sets of $f$ to sublevel sets of $F$, which is why this remains a counterexample.

Now, the justification we promised. To prove that $f^{-1}(-\infty,a]$ is contractible for $a>0$, we need a lemma.

\begin{lemma} \label{kontraktibilnost}
Let $X\subseteq\mathbb R^n$ be contractible. Let $Y=\bigcup_{x\in X}\{x\}\times I_x$, where for each $x$, $I_x$ is a non-empty (but possibly degenerate) interval. Suppose there exists a continuous function $s:X\to\mathbb R$ such that the graph $\Gamma_s=\{(x,s(x))|\;x\in X\}$ is a subset of $Y$. Then $Y$ is contractible.
\end{lemma}

\begin{proof}
Since the ``vertical slices'' $\{x\}\times I_x$ of $Y$ are intervals, the formula $R((x,y),t)=(x,(1-t)y+t s(x))$ defines a map $R:Y\times I\to Y$. This map is a deformation retraction of $Y$ to $\Gamma_s$. The latter is homeomorphic to $X$, and thus contractible.
\end{proof}

This lemma will also be useful to us in the next section. Both times, we use it for the case where $X$ is an interval.

The lemma tells us that to show that the sublevel sets for $a>0$ are contractible, we only need to show that their vertical slices are intervals and exhibit an appropriate continuous function $s$. Before doing this, we establish another lemma.

\begin{lemma} \label{intervali}
Suppose $f,g:X\to\mathbb R$ are continuous functions, such that $f\leq g$ and $X$ is a connected topological space. Then
\begin{equation*}
J=\bigcup_{w\in X}[f(w),g(w)]
\end{equation*}
is an interval.
\end{lemma}

\begin{proof}
Clearly, $f(X)\subseteq J$ and $g(X)\subseteq J$. Since $f$ and $g$ are continuous functions and $X$ is connected, these are intervals. Now let $w\in X$ be arbitrary. Then $[f(w),g(w)]\subseteq J$ and $f(w)\in f(X)$ and $g(w)\in g(X)$. This implies that
\begin{equation*}
(\inf f(X),\sup g(X))\subseteq J\subseteq\overline{(\inf f(X),\sup g(X))}
\end{equation*}
Thus, indeed, $J$ is an interval.
\end{proof}

Now we can prove

\begin{proposition}
Suppose $a>0$ and let $f$ be the function in the example we have just presented. Then
\begin{equation*}
L_a:=f^{-1}(-\infty,a]=\{(x,y)\in\mathbb R|\; d((x,y),L)\leq a\}
\end{equation*}
is contractible.
\end{proposition}

\begin{proof}
First, we show that for any $b\in\mathbb R$, the set $L_a\cap(\{b\}\times\mathbb R)$ is an interval. (Possibly empty or degenerate.) To see this, note that $(b,w)\in\bar B((x,y),a)\cap(\{b\}\times\mathbb R)$ (where by $\bar B(x,r)$ we denote the closed ball of radius $r$ centered at $x$ and $B(x,r)$ is the corresponding open ball) if and only if
\begin{equation*}
y-\sqrt{a^2-(x-b)^2}\leq w\leq y+\sqrt{a^2-(x-b)^2}
\end{equation*}
as an elementary calculation shows. Noting that this square root is only defined for $x\in[b-a,b+a]$, this means that
\begin{equation*}
L_a\cap(\{b\}\times\mathbb R) = \bigcup_{(x,y)\in A}\{b\}\times\left[y-\sqrt{a^2-(x-b)^2},y+\sqrt{a^2-(x-b)^2}\right]
\end{equation*}
where $A=L\cap([b-a,b+a]\times\mathbb R)$.

But $A$ is connected (by an easy exercise) and $(x,y)\to y-\sqrt{a^2-(x-b)^2}$ and $(x,y)\to y+\sqrt{a^2-(x-b)^2}$ are continuous functions on $A$. Therefore, by Lemma \ref{intervali}, $L_a\cap(\{b\}\times\mathbb R)$ is an interval.

To complete our justification, note that $L_a\subseteq[-2-a,1+a]\times\mathbb R$ and for every $b\in[-2-a,1+a]$, the intersection $L_a\cap(\{b\}\times\mathbb R)$ is non-empty. We can therefore define the following continuous function $s:[-2-a,1+a]\to\mathbb R$:
\begin{equation*}
s(x)=\begin{cases}0;&x\geq 1\\
\sin\frac{\pi}x;&x\in[a,1]\\
\sin\frac{\pi}a;&x\leq a.
\end{cases}
\end{equation*}
and the graph $\Gamma_s$ is a subset of $L_a$. The proof is complete by Lemma \ref{kontraktibilnost}.
\end{proof}

We might further want to compactify our example to show that the critical value lemma may fail for compact topological manifolds, but as the next section shows an even stronger result, we omit this.

\section{A compact smooth counterexample}

In this section, we construct a compact smooth manifold embedded in $\mathbb R^4$ which, together with its height function, provides us with yet another counterexample to the critical value lemma. First, we examine the function $f:\mathbb R^2\to\mathbb R$ defined by
\begin{equation*}
f(x,y)=\begin{cases}e^{-\frac1{x^2}}(y-\sin\frac{\pi}x)^2;&x\neq 0\\
0;&x=0.
\end{cases}
\end{equation*}
This is a smooth function (for the proof, see Appendix). It is of interest to us because of its sublevel sets: $f^{-1}(-\infty,0]$ is a connected set with three path components. (It is just a slightly modified version of the topologist's sine curve.) In fact, the level sets of $f$ can all be expressed explicitly. We leave this as an easy exercise for the reader and only state the results here:

\begin{itemize}
\item for $a>0$ we have
\begin{equation*}
f(x,y) = a \iff y = \sin\frac\pi x\pm\sqrt{a}e^{\frac{1}{2x^2}},
\end{equation*}
\item for $a=0$
\begin{equation*}
f(x,y) = 0 \iff y = \sin\frac\pi x \text{ or } x=0,
\end{equation*}
\item and for $a<0$ the level set is empty.
\end{itemize}

We should mention that here and throughout this paper, we consider an equation of the form $y = F(x)$ to be false for points $x$ where $F(x)$ is undefined. In the present case, this only affects the point $x=0$, where $\sin\frac{\pi}x$ is undefined (and is the reason for treating the case $a = 0$ separately). This also gives us the sublevel sets:

\begin{itemize}
\item for $a>0$ we have
\begin{equation*}
f(x,y)\leq a \iff \sin\frac\pi x-\sqrt{a}e^{\frac{1}{2x^2}} \leq y \leq \sin\frac\pi x+\sqrt{a}e^{\frac{1}{2x^2}} \textrm{ or } x=0,
\end{equation*}
which might be written more compactly as
\begin{equation*}
f(x,y)\leq a \iff \left|y - \sin\frac\pi x\right|\leq \sqrt{a}e^{\frac{1}{2x^2}} \textrm{ or } x=0,
\end{equation*}
\item for $a=0$
\begin{equation*}
f(x,y)\leq 0 \iff y = \sin\frac\pi x \textrm{ or } x = 0,
\end{equation*}
\item and for $a<0$ the sublevel set is empty.
\end{itemize}

This is very close to what we are looking for: for $a>0$ the sublevel sets are contractible and the sublevel set for $a=0$ is not contractible. The only things missing are contractible sublevel sets for $a<0$ and compactness of sublevel sets. (Which will later help us make our example compact.) We need to modify $f$ a bit. The following function $\vartheta:\mathbb R\to\mathbb R$ is helpful in doing this:
\begin{equation*}
\vartheta(x)=\begin{cases} e^{-\frac1{x^2}};&x>0\\
0;&x\leq 0.
\end{cases}
\end{equation*}
It is a standard result that this is a smooth function (see for example \cite{lee}; for the basic idea also see Appendix, where we prove that $x\mapsto e^{-\frac1{x^2}}\sin\frac{\pi}x$ is smooth). We now modify $f$ as follows: define $g:\mathbb R^2\to\mathbb R$ by
\begin{equation*}
g(x,y) = f(x,y) - \vartheta\left(\frac14-(x-1)^2\right) + x^2\vartheta(x^2-4)
\end{equation*}
Here, the first correction term gives us the contractible sublevel sets for $a<0$ we wanted (without destroying nice properties we already have) and the second correction term (which conveniently grows to infinity as $x\to\pm\infty$) takes us one step closer to compactness, since the sublevel sets now become bounded in one direction. We will show that this is so in a moment. When we are done showing this, we will introduce a third (and last) correction term, which will cause the sublevel sets to become bounded in the other direction. (And thus compact by Heine-Borel, since the sublevel sets of a continuous function are closed.)

All of what we just said follows from the fact that the sublevel sets can again be written explicitly. The calculation is basically the same as for $f$, so we again leave it to the reader. The results are:

\begin{itemize}
\item for $a\neq 0$ we have
\begin{equation*}
g(x,y) = a \iff y = \sin\frac\pi x\pm e^{\frac{1}{2x^2}}\sqrt{a+\vartheta\left(\frac14-(x-1)^2\right)-x^2\vartheta(x^2-4)},
\end{equation*}
\item and for $a=0$
\begin{equation*}
g(x,y) = 0 \iff y = \sin\frac\pi x\pm e^{\frac{1}{2x^2}}\sqrt{\vartheta\left(\frac14-(x-1)^2\right)-x^2\vartheta(x^2-4)} \text{ or } x=0.
\end{equation*}
\end{itemize}

The results are a bit more complicated than for $f$, so we have to say a few words about them. To understand what these equations tell us, we have to understand where the appropriate expressions are defined.  In particular, the square root is only defined for non-negative arguments. (Note also that the problem with $x=0$ we mentioned when studying the sublevel sets of $f$ remains the same.) This tells us that the equations above only make sense for
\begin{equation*}
a+\vartheta\left(\frac14-(x-1)^2\right)-x^2\vartheta(x^2-4)\geq 0.
\end{equation*}

First, we consider the case $a\geq 0$. In this case, the inequality clearly holds for $x\in[-2,2]$. (As $x^2\vartheta(x^2-4)$ is zero there.) For $x\notin[-2,2]$ the inequality is equivalent to $x^2\vartheta(x^2-4)\leq a$. This inequality is satisfied if and only if $x\in[-\tau(a),\tau(a)]$, where $\tau(a)=\max\{x\in\mathbb R|\; x^2\theta(x^2-4)\leq a\}$. Note that the function $\tau:[0,\infty)\to\mathbb R$ is strictly increasing and $\tau(0)=2$. To sum up: the level sets for $a>0$ are defined (by the above equations) for $x\in[-\tau(a),\tau(a)]\setminus\{0\}$ and the level set for $a = 0$ is defined (by the above equations) for $x\in[-2,2]$.

The remaining case is $a<0$. In this case, the inequality is clearly false for $x\notin(\frac12,\frac32)$. For $x\in(\frac12,\frac32)$ it is equivalent to $a\geq -e^{-\frac{1}{(\frac14-(x-1)^2)^2}}$. As it turns out (an easy exercise which we leave to the reader), for $a\geq-e^{-16}$, this inequality is equivalent to $x\in[1-\sigma(a),1+\sigma(a)]$, where $\sigma(a) = \sqrt{\frac14 - \sqrt{-\frac1{\log(-a)}}}$, which is defined if and only if $a\in[-e^{-16},0)$.

For the sake of completeness, we note that for $a<-e^{-16}$, the inequality implies $-e^{-16}>e^{-\frac{1}{(\frac14-(x-1)^2)^2}}$ which is equivalent to $(\frac14-(x-1)^2)^2>\frac1{16}$. This clearly has no solutions in $(\frac12,\frac32)$. This means that for $a<-e^{-16}$, the level (and sublevel) sets are empty.

The sublevel sets of $g$ are therefore given by:

\begin{itemize}
\item if $a\geq 0$
\begin{equation*}
g(x,y)\leq a \iff \left|y - \sin\frac\pi x\right| \leq e^{\frac{1}{2x^2}}\sqrt{a+\vartheta\left(\frac14-(x-1)^2\right)-x^2\vartheta(x^2-4)} \textrm{ or } x=0,
\end{equation*}
\item if $-e^{-16}\leq a<0$
\begin{equation*}
g(x,y)\leq a\iff \left|y - \sin\frac\pi x\right| \leq e^{\frac{1}{2x^2}}\sqrt{a+\vartheta\left(\frac14-(x-1)^2\right)-x^2\vartheta(x^2-4)},
\end{equation*}
\item empty for $a<-e^{-16}$.
\end{itemize}

In other words, the sublevel sets are the areas bounded by the corresponding level sets. The main point of all this is of course that the sublevel sets retain all the properties we need (namely, they are contractible for all $a\in[-e^{-16},\infty)$, except for $a=0$).

A picture is worth a thousand words, so we present some Mathematica plots of the sublevel sets of $g$.

\begin{itemize}
\item For $a=1$:

\begin{center}
\includegraphics[width=250pt]{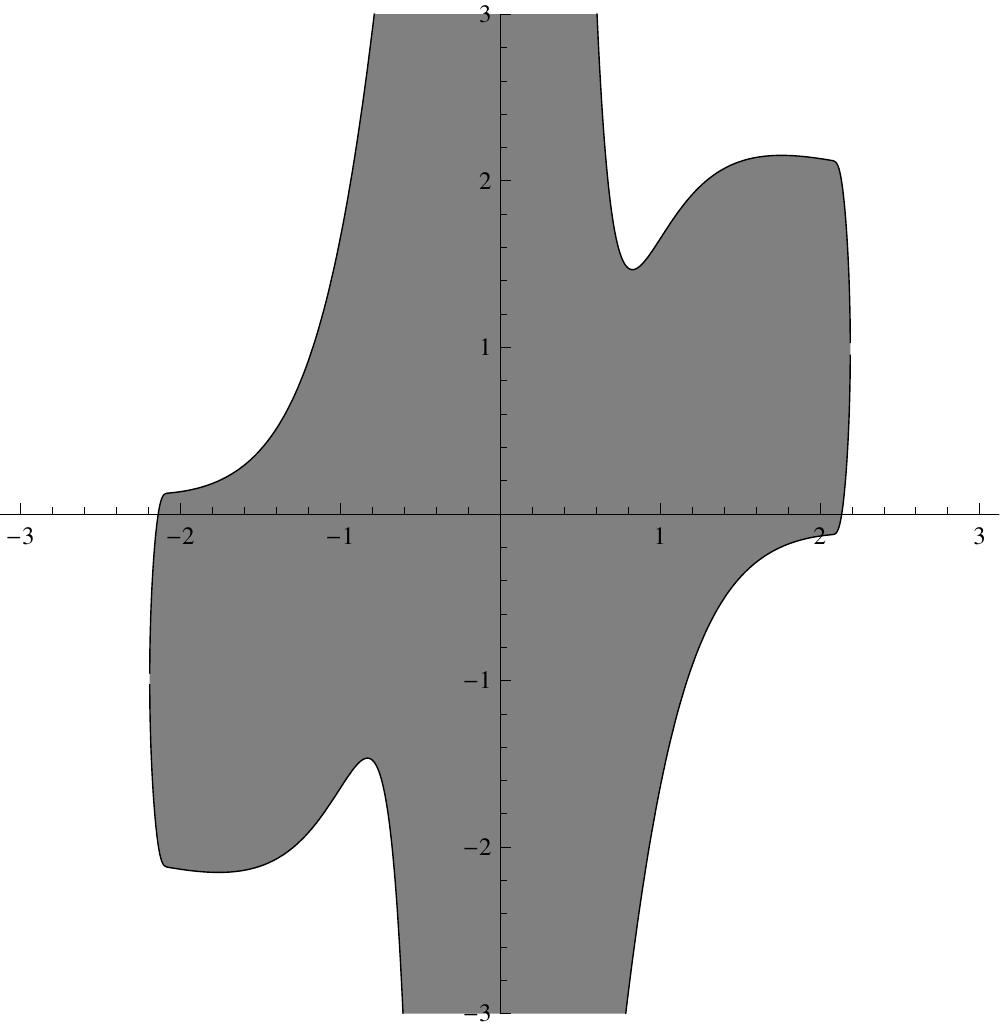}
\end{center}
%\pagebreak
\item For $a=0.01$:

\begin{center}
\includegraphics[width=250pt]{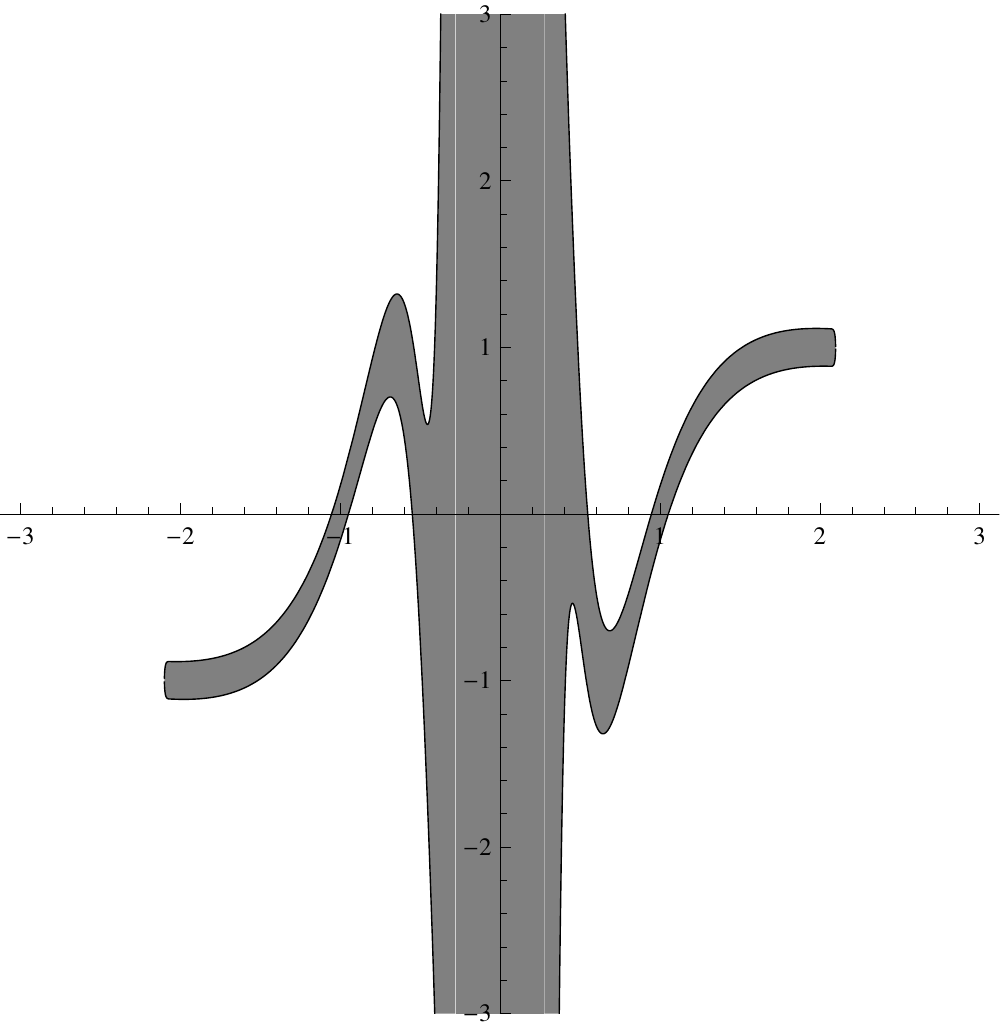}
\end{center}

\item For $a=0$; this one looks like a curve. That's a bit misleading though, since it actually contains a tiny bulge (a topological disk) in the area $(\frac12,\frac32)\times\mathbb R$. This disk is invisible to the naked eye, though. (Actually, the previous two examples also contain invisible bulges, but in that case, these don't affect the topology. So, even though a picture is worth a thousand words, some of those words might be lies.) The bulges mentioned are invisible because $\vartheta(\frac14-(x-1)^2)$ has very small values and thus causes only a tiny perturbation on the appropriate sublevel set of $f$ (which is indeed a curve for $x\in(\frac12,\frac32)$):

\begin{center}
\includegraphics[width=250pt]{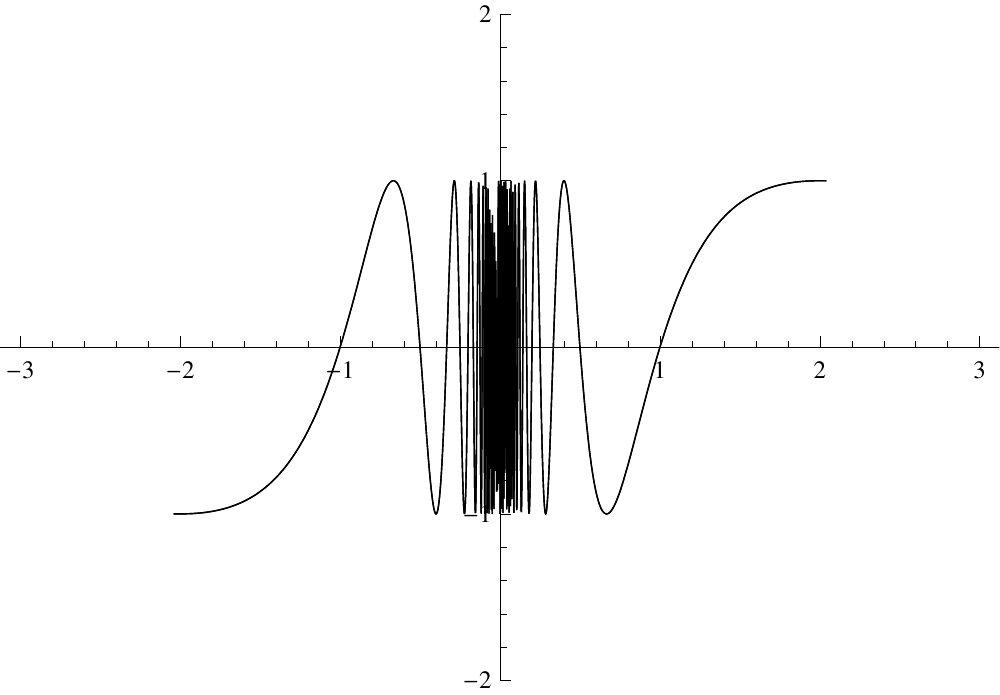}
\end{center}

%\pagebreak
\item For $a<0$; this one should be completely invisible to the naked eye. The following plot is a sublevel set for $a<0$ where instead of the correction term $\vartheta(\frac14-(x-1)^2)$ we subtracted $10^7\vartheta(\frac14-(x-1)^2)$. This isn't a sublevel set of $g$, but we hope it should convey to the reader the same general idea about how the correction term works:

\begin{center}
\includegraphics[width=250pt]{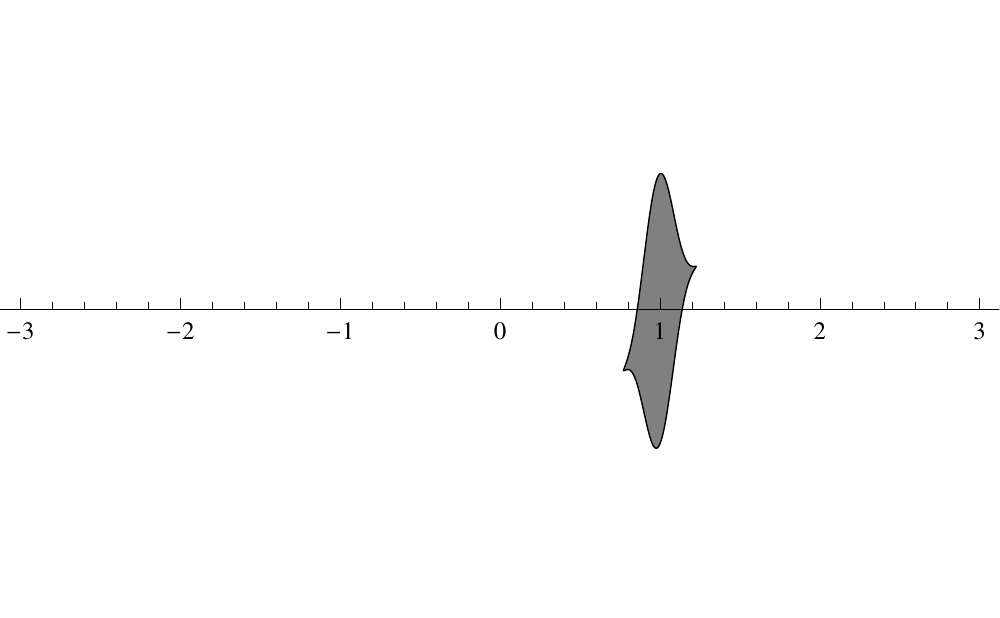}
\end{center}
\end{itemize}

Now we introduce the final correction term. We define the function $h:\mathbb R^2\to\mathbb R$ by the formula
\begin{equation*}
h(x,y) = f(x,y) - \vartheta\left(\frac14-(x-1)^2\right) + x^2\vartheta(x^2-4) + y^2\vartheta(y^2-4)
\end{equation*}
Because of this last term, the function now ``diverges to infinity as $(x,y)$ approaches infinity''. This means that the sublevel sets are now bounded and thus compact by Heine-Borel. We only need to show now that they retain all the nice properties they had before.

This time, we won't be able to describe all the sublevel sets explicitly. Note, however, that the sublevel (and level) sets for $a<0$ remain unchanged, as the new correction term has support in $\mathbb R\times(\mathbb R\setminus(-2,2))$, which these sublevel sets don't intersect. The level set $h^{-1}(\{0\})$ also has a nice description:
%\begin{multline*}
%h(x,y) = 0 \iff y = \sin\frac\pi x\pm e^{\frac{1}{2x^2}}\sqrt{\vartheta\left(\frac14-(x-1)^2\right)-x^2\vartheta(x^2-4)} \text{ or }\\x=0 \text{ and } y\in[-2,2].
%\end{multline*}
\begin{equation*}
h(x,y) = 0 \iff y = \sin\frac\pi x\pm e^{\frac{1}{2x^2}}\sqrt{\vartheta\left(\frac14-(x-1)^2\right)-x^2\vartheta(x^2-4)} \text{ or } (x=0 \text{ and } y\in[-2,2]).
\end{equation*}
This is because the only points in $g^{-1}(\{0\})$ where the value doesn't remain $0$ are those in $\mathbb R\times(\mathbb R\setminus[-2,2])$. Here the value of $g$ was already non-negative (elementary calculation), and the correction term is strictly positive. Therefore positive values are assigned to every point in $\mathbb R\times(\mathbb R\setminus[-2,2])$. Again, we have to keep in mind that $\sqrt{\vartheta(\frac14-(x-1)^2)-x^2\vartheta(x^2-4)}$ only makes sense for $x\in[-2,2]$. From this we can also explicitly describe the sublevel set for $a=0$.

In any case, this tells us that for $-e^{-16}\leq a<0$ the set $h^{-1}(-\infty,a]$ is contractible and that $h^{-1}(-\infty,0]$ has three (contractible) path components. (And a single connected component.)

The level sets for $a>0$ are a bit harder to describe. Before plunging into proofs, we again provide some pictures. The parameters are the same as for $g$, so that the reader may compare the differences. (Note that we left out the sublevel sets for non-positive values of $a$, since these would tell us nothing new.)
%\pagebreak
\begin{itemize}
\item For $a=1$:

\begin{center}
\includegraphics[width=250pt]{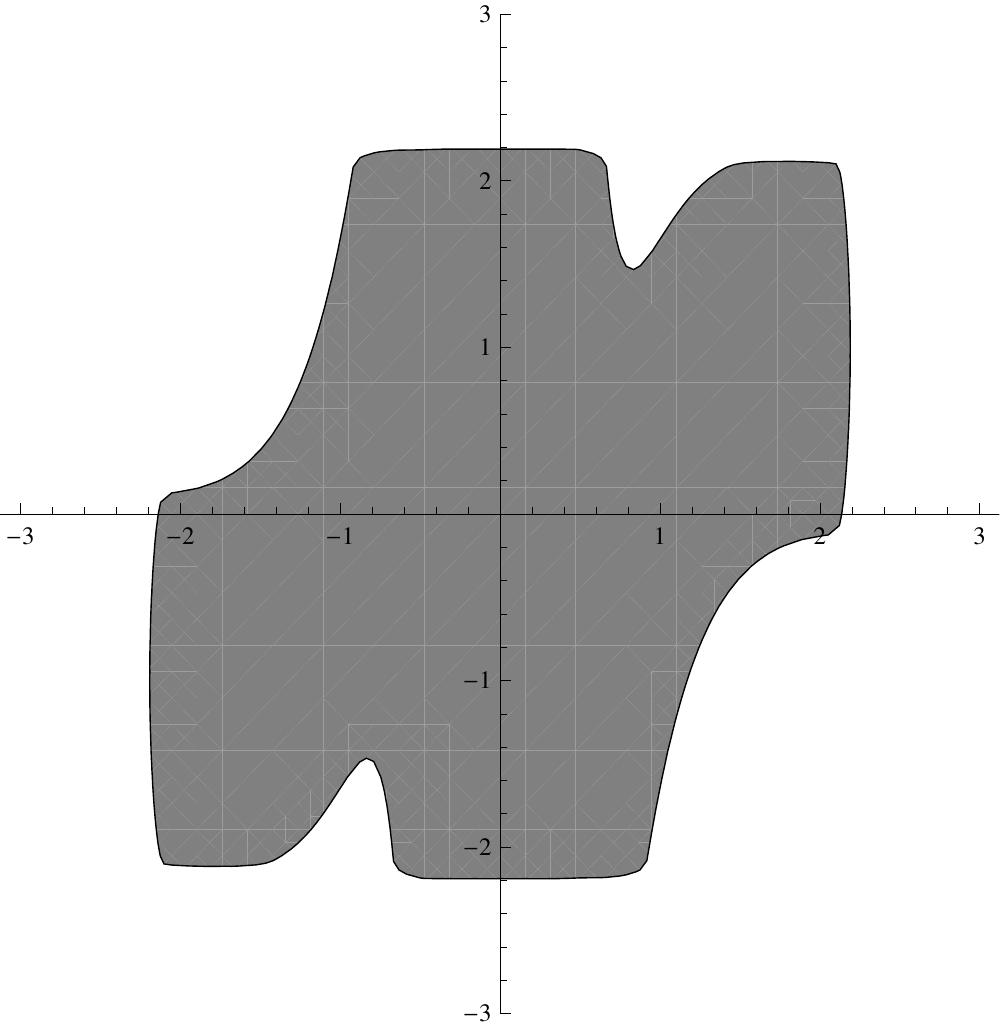}
\end{center}

\item For $a=0.01$:

\begin{center}
\includegraphics[width=250pt]{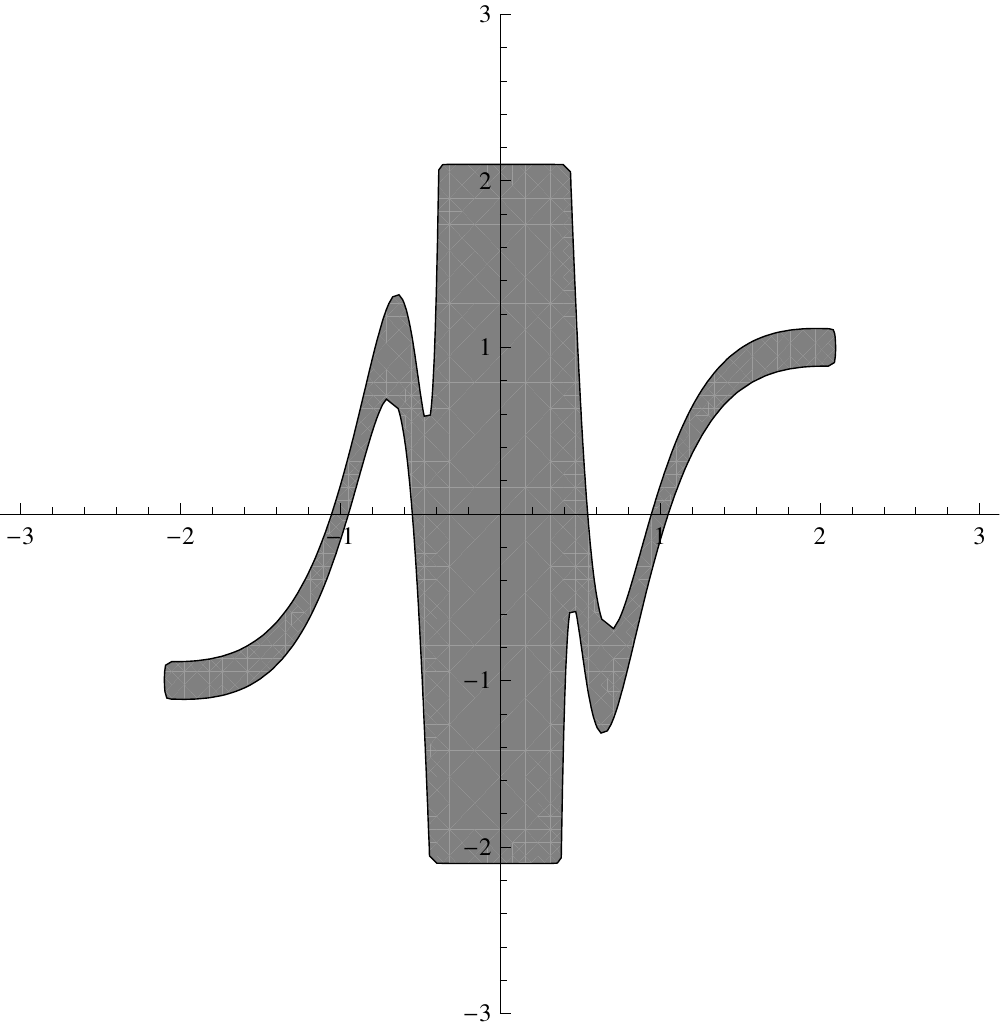}
\end{center}
\end{itemize}

We establish contractibility of $h^{-1}(-\infty,a]$ in two propositions.

\begin{proposition}
Suppose $a>0$. Then the vertical slices $(\{x\}\times\mathbb R)\cap h^{-1}(-\infty,a]$ are of the form $\{x\}\times I_x$, where $I_x$ is a (possibly empty or degenerate) interval. The set $X=\{x\in\mathbb R|\;I_x \text{ is non-empty}\}$ is also an interval.
\end{proposition}

\begin{proof}
We examine the behaviour of $h$ along the vertical lines $\{x\}\times\mathbb R$. To do this, it is convenient to define for each $x\in\mathbb R$ a function $h_x:\mathbb R\to\mathbb R$ by $h_x(y)=h(x,y)$. This is a smooth function. It is immediate from the definition of $h$ that $h_x$ achieves a global minimum at $y=\sin\frac\pi x$ for all $x\neq 0$. An elementary calculation verifies that for $x\neq 0$, the function $h_x$ is strictly increasing on $[\sin\frac\pi x,\infty)$ and strictly decreasing on $(-\infty,\sin\frac\pi x]$. (Note that here we rely on the fact that the last correction term is zero for $y\in[-2,2]$, which means that this term can be understood as an increasing function of $y$ on $[\sin\frac\pi x,\infty)$ and decreasing on $(-\infty,\sin\frac\pi x]$.)

For $x=0$, $h_x$ behaves a bit differently: $h_x(y)=y^2\vartheta(y^2-4)$ holds for all $y\in\mathbb R$. Thus it achieves its global minimum (zero) at every $y\in[-2,2]$. It is strictly increasing on $[2,\infty)$ and strictly decreasing on $(-\infty,-2]$.

The first part of the claim follows.

For the second part of the claim, note that $X$ contains precisely the points $x$ for which $\mu_x:=\min_{y\in\mathbb R} h_x(y)\leq a$, i.e. $X$ is precisely the sublevel set $\mu^{-1}(-\infty,a]$ of the function $\mu:x\mapsto\mu_x$. Now, note that
\begin{equation*}
\mu_x=-\vartheta\left(\frac14-(x-1)^2\right)+x^2\vartheta(x^2-4)
\end{equation*}
holds by the properties of $h_x$ we described above. Therefore $\mu$ is increasing for $x\geq 1$ and decreasing for $x\leq 1$ (both times non-strictly). Therefore the sublevel sets of $\mu$ are intervals. This completes the proof.
\end{proof}

We can now show that the sublevel sets for $a>0$ are contractible.

\begin{proposition}
For $a>0$, $h^{-1}(-\infty,a]$ is contractible.
\end{proposition}

\begin{proof}
First note that $h^{-1}(-\infty,a)$ is an open set containing $h^{-1}(-\infty,0]$. Therefore by the tube lemma (\cite{tube}), it must also contain a tube around $\{0\}\times[-2,2]$, i.e. an open set of the form $(\alpha,\beta)\times(\gamma,\delta)\supseteq\{0\}\times[-2,2]$. (This means, of course, that $\alpha<0<\beta$.) The previous proposition tells us that the sublevel set $h^{-1}(-\infty,a]$ is precisely $\bigcup_{x\in X}\{x\}\times I_x$, for some contractible space $X$. As for the topological manifold counterexample, contractibility will follow by defining a function $s:X\to\mathbb R$ as follows:
\begin{equation*}
s(x)=\begin{cases}\sin\frac\pi x;& x\notin(\alpha,\beta)\\
\sin\frac\pi \alpha +\frac{x-\alpha}{\beta-\alpha}(\sin\frac\pi \beta - \sin\frac\pi \alpha);& x\in[\alpha,\beta].
\end{cases}
\end{equation*}
This is a continuous function, and its graph $\Gamma_s=\{(x,s(x))|\;x\in X\}$ is a subset of $h^{-1}(-\infty,a]$. This is because $h(x,\sin\frac\pi x) =\mu_x\leq a$ for $x\in X\setminus\{0\}$ and for $x\in(\alpha,\beta)$, the point $(x,s(x))$ lies on the open line segment connecting the points $(\alpha,\sin\frac\pi\alpha)$ and $(\beta,\sin\frac\pi\beta)$ and thus in the tube $(\alpha,\beta)\times(\gamma,\delta)\subseteq h^{-1}(-\infty,a]$. By Lemma \ref{kontraktibilnost}, $h^{-1}(-\infty,a]$ is contractible.
\end{proof}

This combined with the contractibility of sublevel sets for $\-e^{-16}\leq a<0$ and non-contractibility for $a=0$ proves that $h:\mathbb R^2\to\mathbb R$ is indeed a counterexample to the critical value lemma.

Making the counterexample compact is easy now. Let $S^2 =\{(x,y,z)\in\mathbb R^3|\; x^2+y^2+z^2 = 1\}$. As we know, $S^2\setminus\{(0,0,1)\}$ is diffeomorphic to $\mathbb R^2$ via the stereographic projection $p:S^2\setminus\{(0,0,1)\}\to\mathbb R^2$. The set $K=(h\circ p)^{-1}(-\infty,1]$ is a compact subset of $S^2\setminus\{(0,0,1)\}$. Let $U\subseteq S^2$ be an open neighborhood of $(0,0,1)$ such that $\bar U\subseteq S^2\setminus K$. As we know from the theory of smooth manifolds (see \cite[Chapter 2]{lee} for example) there exists a smooth function $\eta:S^2\to\mathbb R$ such that $\eta(x,y,z) = 0$ for all $(x,y,z)\in U$ and $\eta(x,y,z)=1$ for $(x,y,z)\in K$. This means we can define a smooth function $H:S^2\to\mathbb R$ as follows:
\begin{equation*}
H(w)=\begin{cases}(h\circ p)(w)\eta(w)+(1-\eta(w));&w\neq(0,0,1),\\
1;&w=(0,0,1).
\end{cases}
\end{equation*}
Clearly, $H$ is a smooth function. It is smooth on $S^2\setminus\{(0,0,1)\}$ because $h,p$ and $\eta$ are smooth and it is smooth in $(0,0,1)$ because $H(w)=1$ for all $w\in U$. The sublevel sets of $H$ are diffeomorphic to those of $h$ for $a\in[-e^{-16},1)$. (Since in our construction of $H$ we only modified those values of $h\circ p$ that were greater than $1$ and the modified values remain greater than or equal to $1$.)

To make the counterexample a height function on a smooth manifold, we notice that $i:(x,y,z)\to(x,y,z,H(x,y,z))$ is a smooth embedding of $S^2$ into $\mathbb R^4$. We define the height function $F:i(S^2)\to\mathbb R$ by $F(x,y,z,w)=w$ and the sublevel sets of $F$ are diffeomorphic to those of $H$ (and $h$). Conclusion: we have found a compact smooth manifold embedded in $\mathbb R^4$ and a (smooth, of course) height function on it that contradicts the critical value lemma.

\section{Positive results}

To provide some contrast to all the counterexamples above, we devote this section to proving a positive result about symmetric homological critical values. As the result is a version of the critical value lemma valid for symmetric homological critical values, we suggest the name ``symmetric critical value lemma''.

Throughout this section, let $X$ denote an arbitrary topological space and let $f:X\to\mathbb R$ be an arbitrary continuous function on $X$. As we have a lot to say about sublevel sets, it is convenient to introduce some notation. We define for each $a\in\mathbb R$ the following subspaces of $X$:

\begin{itemize}
\item $X_a=f^{-1}(-\infty,a]$, which we refer to as the closed sublevel set of $f$ at $a$,
\item $X_a^-=f^{-1}(-\infty,a)$, which we refer to as the open sublevel set of $f$ at $a$.
\end{itemize}

To further simplify notation, we also introduce the following notations for the various inclusions of sublevel sets:

\begin{itemize}
\item $i_a^b:X_a\hookrightarrow X_b$,
\item $i_{a-}^b:X_a^-\hookrightarrow X_b$,
\item $i_a^{b-}:X_a\hookrightarrow X_b^-$,
\item $i_{a-}^{b-}:X_a^-\hookrightarrow X_b^-$,
\end{itemize}

whenever the space on the left is a subspace of the space on the right. Our result is the following.

\begin{theorem}
Suppose the function $f:X\to\mathbb R$ has no symmetric homological critical values on the interval $[x,y)$. Then the inclusion $X_x^-\hookrightarrow X_y^-$ induces isomorphisms on all homology groups.
\end{theorem}

Note that this result holds for spaces that may be immensely more pathological than the spaces the usual critical value lemma is used to handle. In particular, the result holds for the counterexamples we have presented. We thus expect it to be useful in studying such pathological spaces.

We prove our result in four steps. Throughout the rest of this section, let $k$ be a fixed integer.

\begin{remark}
We will also assume throughout this section that the intervals $[a,b)$ we are working with are non-degenerate, i.e. $a<b$.
\end{remark}

We start with the following:

\begin{step} \label{one}
Suppose $a<c<b<d$ are real numbers such that the inclusions $i_a^b:X_a\hookrightarrow X_b$ and $i_c^d:X_c\hookrightarrow X_d$ induce isomorphisms on $H_k$. Then the inclusions $i_a^c,i_c^b$ and $i_b^d$ also induce isomorphisms on $H_k$. (A completely analogous result holds also for any inclusions of the forms $i_{a-}^b,i_a^{b-}$ and $i_{a-}^{b-}$.)
\end{step}

\begin{proof}
The homomorphism $(i_c^d)_*:H_k(X_c)\to H_k(X_d)$ induced by the inclusion $i_c^d$ is an isomorphism. Therefore, $(i_c^b)_*$ is a monomorphism and $(i_b^d)_*$ is an epimorphism. Likewise, because $(i_a^b)_*$ is an isomorphism, $(i_a^c)_*$ is a monomorphism and $(i_c^b)_*$ is an epimorphism. Summing up, $(i_c^b)_*$ is an isomorphism. This immediately implies that $(i_a^c)_*=(i_c^b)_*^{-1}(i_a^b)_*$ and $(i_b^d)_*=(i_c^d)_*(i_c^b)_*^{-1}$ are isomorphisms. (The proof for the other three kinds of inclusions is completely analogous and we leave it to the reader.)
\end{proof}

As the next step, we prove:

\begin{step}
Suppose an interval $[x,y)$ contains no symmetric homological critical values of $f:X\to\mathbb R$. Let $\delta\in(0,y-x)$. Then there exist real numbers $\epsilon,\eta\in(0,\delta)$ such that the homomorphism $(i_{x-\epsilon}^{y-\eta})_*$ is an isomorphism.
\end{step}

\begin{proof}
Since there are no symmetric homological critical values on $[x,y-\delta]$, we may choose for each $z\in[x,y-\delta]$ an $\epsilon(z)\in(0,\frac\delta 2)$ such that $(i_{z-\epsilon(z)}^{z+\epsilon(z)})_*$ is an isomorphism. The interval $[x,y-\delta]$, however, is compact, so we may choose $z_1,z_2,\ldots,z_n\in[x,y-\delta]$ for some $n\in\mathbb N$ such that
\begin{equation*}
[x,y-\delta]\subseteq\bigcup_{i=1}^n(z_i-\epsilon(z_i),z_i+\epsilon(z_i))\subseteq \left(x-\frac\delta 2,y-\frac\delta 2\right)
\end{equation*}
holds. The latter inclusion holds because $\epsilon(z)\in(0,\frac\delta 2)$ for all $z$. This basically completes the proof. From a $n$-element cover $\{(a_i,b_i)|\;i=1,2,\ldots,n\}$ satisfying
\begin{equation*}
[x,y-\delta]\subseteq\bigcup_{i=1}^n(a_i,b_i)\subseteq \left(x-\frac\delta 2,y-\frac\delta 2\right)
\end{equation*}
and where $n>1$ and $(i_{a_i}^{b_i})_*$ are isomorphisms, we may easily construct a $(n-1)$-element cover satisfying the same equation (where we replace $n$ with $n-1$, of course) such that the corresponding induced maps are again isomorphisms. To do this, simply choose two elements of this cover with non-empty intersection, say $(a,b)$ and $(c,d)$. If one is entirely contained in the other, throw it away and we are done. Otherwise, we may assume without loss of generality that $a<c$ and use the first step to see that $(i_a^d)_*$ is an isomorphism. Thus, replacing $(a,b)$ and $(c,d)$ with $(a,d)$ gives us the desired $(n-1)$-element cover. Repeating this procedure $n-1$ times, we are left with a cover containing a single element $(a_1,b_1)$. It satisfies the equation above, i.e. $[x,y-\delta]\subseteq(a_1,b_1)\subseteq (x-\frac\delta 2,y-\frac\delta 2)$ and $(i_{a_1}^{b_1})_*$ is an isomorphism. This means that we can define $\epsilon=x-a_1$ and $\eta=y-b_1$, thus concluding the proof.
\end{proof}

Now, we need a classical lemma.

\begin{lemma} \label{hatcher}
Suppose $X$ is a topological space such that $X = \bigcup_{n\in\mathbb N} Z_n$, where $(Z_n)_n$ is an increasing (i.e. $Z_n\subseteq Z_m$ for $n<m$) sequence of subsets of $X$ with the property that for every compact set $K\subseteq X$ there exist a $n\in\mathbb N$ such that $K\subseteq Z_n$. Then the canonical map $\varinjlim\limits_{n\in\mathbb N} H_i(Z_n)\to H_i(X)$ is an isomorphism for all $i\in\mathbb N_0$. (Here $\varinjlim\limits_{n\in\mathbb N} H_i(Z_n)$ is the direct limit of the direct system of homology groups and homomorphisms induced by inclusions.) In particular, the conclusion holds if the sets $Z_n$ are open.

In the special case when all the induced homomorphisms are isomorphisms, the conclusion is stronger: for each $j\in\mathbb N$ the direct limit $\varinjlim\limits_{n\in\mathbb N} H_i(Z_n)$ is canonically isomorphic to $H_i(Z_j)$ (together with the obvious maps) and the canonical isomorphism $H_i(Z_j)\to H_i(X)$ (guaranteed by the first part of the claim) is precisely the map induced by the inclusion.
\end{lemma}

\begin{proof}
For the first part of the claim, see \cite[Proposition 3.33]{hatcher}. If the sets $Z_n$ are open, they form an open cover for $X$ (and its compact subsets). Thus in that case each compact set is contained in some $Z_n$.

For the second part of the claim we notice that for every $n\in\mathbb N$ there is an isomorphism $H_i(Z_n)\to H_i(Z_j)$, either induced by inclusion or an inverse of such. The group $H_i(Z_j)$ together with these isomorphisms satisfies the universal property of the direct limit and is thus canonically isomorphic to it. Therefore there is a unique map $H_i(Z_j)\to H_i(X)$ that makes the appropriate diagram commute and by the first part of the claim it is an isomorphism. Since the map $H_i(Z_j)\to H_i(X)$ induced by inclusion makes the mentioned diagram commute, it must be the same map and thus also an isomorphism.
\end{proof}

We proceed to the next step.

\begin{step}
Suppose an interval $[x,y)$ contains no symmetric homological critical values of $f$. Let $\delta\in(0,y-x)$. Then there is a real number $\epsilon\in(0,\delta)$ such that $(i_{x-\epsilon}^{y-})_*$ is an isomorphism.
\end{step}

\begin{proof}
By the preceding step, we may choose for each $n\in\mathbb N$ real numbers $\epsilon_n,\eta_n\in(0,\min\{\delta,\frac1n\})$ such that $(i_{x-\epsilon_n}^{y-\eta_n})_*$ is an isomorphism. Without loss of generality (passing to subsequences twice if necessary) we may assume that the sequences $(\epsilon_n)_n,(\eta_n)_n$ are strictly decreasing. The map $(i_{x-\epsilon_1}^{y-\eta_1})_*$ is an isomorphism, and (by the first step) so are the maps $(i_{y-\eta_n}^{y-\eta_{n+1}})_*$ for $n\in\mathbb N$. All these maps are induced by inclusions and form a direct system as in the hypotheses of Lemma \ref{hatcher}. Note also that the increasing union $X_y^- = \bigcup_{n=0}^\infty X_{y-\eta_n}$ (where we additionally define $\eta_0=y-x+\epsilon_1$, to avoid having to write the space $X_{x-\epsilon_1}$ separately) has the required property that every compact set $K\subseteq X_y^-$ is contained in some $X_{y-\eta_n}$. (This an easy consequence of the fact that the sets $X_{y-\frac{\eta_n}{2}}^-$ form an open cover of $X_y^-$.) This means we can apply Lemma \ref{hatcher} to this direct system, which shows that the claim holds with $\epsilon=\epsilon_1$.
\end{proof}

The last step, of course, is the lemma itself:

\begin{step}
Suppose an interval $[x,y)$ contains no symmetric homological critical values of $f$. Then $(i_{x-}^{y-})_*$ is an isomorphism.
\end{step}

\begin{proof}
Let $\delta\in(0,y-x)$. By the preceding step, we may choose for each $n\in\mathbb N$ a real number $\epsilon_n\in(0,\min\{\delta,\frac1n\})$ such that $(i_{x-\epsilon_n}^{y-})_*$ is an isomorphism. Without loss of generality we may assume that the sequence $(\epsilon_n)_n$ is strictly decreasing. The map $(i_{x-\epsilon_n}^{x-\epsilon_m})_*=(i_{x-\epsilon_m}^{y-})_*^{-1}(i_{x-\epsilon_n}^{y-})_*$ is thus an isomorphism for every pair $n<m$. This means that the maps $(i_{x-\epsilon_n}^{x-\epsilon_m})_*$ form a direct system of induced maps as in the hypotheses of Lemma \ref{hatcher}. Repeating the argument of the preceding step, we see that $(i_{x-\epsilon_n}^{x-})_*$ is an isomorphism for every $n$. This means that $(i_{x-}^{y-})_* = (i_{x-\epsilon_1}^{y-})_*(i_{x-\epsilon_1}^{x-})_*^{-1}$ is an isomorphism. The proof is complete.
\end{proof}

\begin{remark}
As Primo\v{z} \v{S}kraba points out, an easy consequence of this result is that if $\operatorname{hfs}X>0$ (where $\operatorname{hfs}X$ is defined as in \cite{stability}; i.e. the infimum of positive homological critical values of the distance function $d^X:M\to\mathbb R$ defined by $d^X(y)=d(y,X)$) then the inclusion $(d^X)^{-1}(-\infty,\epsilon]\to(d^X)^{-1}(-\infty,\operatorname{hfs}X)$ induces isomorphisms on all homology groups for all $\epsilon\in(0,\operatorname{hfs}X)$.

(We justify this as follows: the inclusion $(d^X)^{-1}(-\infty,\epsilon)\to(d^X)^{-1}(-\infty,\operatorname{hfs}X)$ induces isomorphisms by the symmetric critical value lemma, and the inclusion $(d^X)^{-1}(-\infty,\frac\epsilon2]\hookrightarrow(d^X)^{-1}(-\infty,\epsilon]$ induces isomorphisms by the critical value lemma. By Step \ref{one} of our proof, $(d^X)^{-1}(-\infty,\epsilon)\hookrightarrow(d^X)^{-1}(-\infty,\epsilon]$ also induces isomorphisms. Therefore any map induced by the above mentioned inclusion is a composition of isomorphisms and thus an isomorphism.)
\end{remark}

\section{Conclusions and future work}

We have shown that the definition of homological critical value as stated in \cite{stability} is problematic. We have presented several counterexamples to the critical value lemma under that definition. Nevertheless, we have also seen that a positive result remains valid for open sublevel sets. This suggests that the open sublevel sets have a very nice behaviour when it comes to homology. Noting that every symmetric homological critical value is a homological critical value, but not vice versa, this positive result might open (pun not intended) up new possibilities of studying somewhat more pathological examples of topological spaces.

Namely, a function may possess many homological critical values (making it impossible to apply the critical value lemma), but it may happen that none of these homological critical values are symmetric. In such a case, we may still apply the symmetric version of the critical value lemma to yield something possibly useful.

As a particular example, note that the symmetric critical value lemma holds for each of the counterexamples we have presented, even though the original critical value lemma cannot be used directly. (In each of these cases there is only one homological critical value in the interval under consideration, but one may easily construct examples without symmetric homological critical values, where the set of homological critical values has limit points.)

\section*{Acknowledgements}

The author would like to thank Jo\~ao Pita Costa, Jaka Smrekar and Primo\v{z} \v{S}kraba for numerous helpful discussions and suggestions. This work was funded by the EU project TOPOSYS (FP7-ICT-318493-STREP).

\section*{Appendix}

Here, we prove some of the facts we have used (and omitted from the main sections) in the construction of the smooth counterexample. Mainly, this section is devoted to proving that the functions mentioned are indeed smooth using standard techniques. We first need the following simple lemma.

\begin{lemma}
Let $f:\mathbb R\to\mathbb R$ be continuous everywhere and differentiable for all $x\in\mathbb R\setminus\{0\}$. Suppose $\lim_{x\to 0}f'(x)=A$. Then $f$ is differentiable everywhere and $f'(0)=A$.
\end{lemma}

\begin{proof}
First, we consider the case where $A=0$ and $f(0)=0$. For every real number $x\neq 0$, the mean value theorem provides us with a real number $\xi_x$ strictly between $0$ and $x$, such that $f'(\xi_x)=\frac{f(x)-f(0)}{x-0} =\frac{f(x)}x$. Let $\epsilon>0$. Since $\lim_{x\to 0}f'(x)=0$, there is $\delta>0$ such that for all $x\in(-\delta,\delta)$ we have $|f'(x)|<\epsilon$. Therefore for $x\in(-\delta,\delta)$ we have $|\frac{f(x)}x|=|f'(\xi_x)|<\epsilon$. Since such a $\delta$ exists for every $\epsilon$, this means precisely that $\lim_{x\to 0}\frac{f(x)-f(0)}{x-0}=0$, so we have proved the special case.

In the general case, we proceed as follows: define a new function $g$ by $g(x) = f(x)-f(0)-Ax$. Then $g$ is continuous everywhere, differentiable on $\mathbb R\setminus\{0\}$, $g(0)=0$ and $\lim_{x\to 0}g'(x)=0$. Therefore $g$ falls into the special case we proved above. This tells us that $g'(0)$ exists and is equal to $0$. But since $f(x)=g(x)+f(0)+Ax$, this means that $f'(0)$ also exists and $f'(0)=A$. This completes the proof.
\end{proof}

Note that this same lemma also applies to functions defined only in a neighborhood of $0$, since continuity and differentiability are local properties. Now, we can prove the following.

\begin{proposition}
The function $f:\mathbb R^2\to\mathbb R$ defined by
\begin{equation*}
f(x,y)=\begin{cases}e^{-\frac1{x^2}}(y-\sin\frac{\pi}x)^2;&x\neq 0\\
0;&x=0
\end{cases}
\end{equation*}
is smooth.
\end{proposition}

\begin{proof}
We first write $f$ in the form $f(x,y) = e^{-\frac1{x^2}}y^2 - 2e^{-\frac1{x^2}}\sin(\frac{\pi}x)y+e^{-\frac1{x^2}}\sin^2\frac{\pi}x$, where $e^{-\frac1{x^2}}$, $e^{-\frac1{x^2}}\sin\frac{\pi}x$ and $e^{-\frac1{x^2}}\sin^2\frac{\pi}x$ are understood to equal $0$ for $x=0$. Under this convention, it suffices to prove that the single variable functions defined by these three expressions are all smooth. The smoothness of the first one is a standard result, and is proven by mathematical induction. We illustrate the method on the second expression. Let $k:\mathbb R\to\mathbb R$ be the function defined by $k(x) = e^{-\frac1{x^2}}\sin\frac{\pi}x$ and $k(0)=0$. The smoothness for $x\neq 0$ follows from analiticity, so it only remains to prove the smoothness for $x=0$. We prove that for each $n\in\mathbb N$ there exist polynomials $P,Q$ such that $k^{(n)}(x)=e^{-\frac1{x^2}}(P(\frac1x)\sin\frac{\pi}x+Q(\frac1x)\cos\frac{\pi}x)$ for $x\neq 0$ and $k^{(n)}(0)=0$. For $n=0$, this is just the definition. (Take $P=1$ and $Q=0$.) Now, suppose we have proved this for some $n$. Then for $n+1$ and $x\neq 0$ we have the following:
\begin{align*}
k^{(n+1)}(x)&=\left(e^{-\frac1{x^2}}\left(P\left(\frac1x\right)\sin\frac{\pi}x+Q\left(\frac1x\right)\cos\frac{\pi}x\right)\right)'\\
&=e^{-\frac1{x^2}}(2x^{-3})\left(P\left(\frac1x\right)\sin\frac{\pi}x+Q\left(\frac1x\right)\cos\frac{\pi}x\right) + e^{-\frac1{x^2}}\bigg(P'\left(\frac1x\right)(-x^{-2})\sin\frac{\pi}x\\
&+P\left(\frac1x\right)\cos\left(\frac{\pi}x\right)(-\pi x^{-2})+Q'\left(\frac1x\right)(-x^{-2})\cos\frac{\pi}x-Q\left(\frac1x\right)\sin\left(\frac{\pi}x\right)(-\pi x^{-2})\bigg)
\end{align*}
This shows that for $n+1$ we may take the polynomials $P^+$ and $Q^+$ defined by $P^+(x)=2x^3P(x)-x^2P'(x)+\pi x^2Q(x)$ and $Q^+(x)=2x^3Q(x)-\pi x^2P(x)-x^2Q'(x)$. Now, note that every function of the form $e^{-\frac1{x^2}}(P(\frac1x)\sin\frac{\pi}x+Q(\frac1x)\cos\frac{\pi}x)$ defined on $\mathbb R\setminus\{0\}$ has a limit as $x\to 0$ and this limit is $0$. (By an easy exercise in real analysis.) Therefore the previous lemma applies to show that $k^{(n+1)}(0)=0$. This completes the induction and proves that $k$ is smooth.

The proof for the other two expressions may be done in a similar manner, for example the derivatives of $e^{-\frac1{x^2}}\sin^2\frac{\pi}x$ can be expressed in the form $e^{-\frac1{x^2}}(P(\frac1x)\sin^2\frac{\pi}x+Q(\frac1x)\sin\frac{\pi}x\cos\frac{\pi}x+R(\frac1x)\cos^2\frac{\pi}x)$ for some polynomials $P,Q,R$.
\end{proof}

\end{document}